\documentclass[11pt,a4paper]{article}
\usepackage{amsthm,amssymb,amsfonts,amsmath,lineno,mathtools,graphicx}
\usepackage{soul}
\usepackage{enumerate}

\usepackage{epsfig,xspace,algorithm,verbatim}
\usepackage{algorithm}
\usepackage{algpseudocode}
\usepackage{caption}
\usepackage{color}
\usepackage{enumerate}
\usepackage{esvect} 
\usepackage{authblk}

\newtheorem{theorem}{Theorem}[section]

\newtheorem{proposition}[theorem]{Proposition}
\newtheorem{corollary}[theorem]{Corollary}
\newtheorem{conjecture}[theorem]{Conjecture}
\newtheorem{definition}[theorem]{Definition}

\newtheorem{open}[theorem]{Open problem}
\newtheorem{claim}[theorem]{Claim}

\newcommand{\bdf}{\begin{df} \begin{rm}}
\newcommand{\edf}{\end{rm} \end{df}}

\makeatletter
\DeclareRobustCommand{\cev}[1]{%
  \mathpalette\do@cev{#1}%
}
\newcommand{\do@cev}[2]{%
  \fix@cev{#1}{+}%
  \reflectbox{$\m@th#1\vec{\reflectbox{$\fix@cev{#1}{-}\m@th#1#2\fix@cev{#1}{+}$}}$}%
  \fix@cev{#1}{-}%
}
\newcommand{\fix@cev}[2]{%
  \ifx#1\displaystyle
    \mkern#23mu
  \else
    \ifx#1\textstyle
      \mkern#23mu
    \else
      \ifx#1\scriptstyle
        \mkern#22mu
      \else
        \mkern#22mu
      \fi
    \fi
  \fi
}

\makeatother

\begin{document}

%
%
%
%

\title{Antimagic and product antimagic graphs with pendant edges}

\author[1]{Merc\`e Mora\thanks{This work was partially supported by 
		grant  Gen. Cat. DGR 2021-SGR-00266 and by grant PID2019-104129GB-I00 funded by MICIU/AEI/10.13039/501100011033
		from the Spanish Ministry of Science and Innovation.}}

\author[2]{Joaqu\'in Tey
}

\affil[1]{Departament de Matem\`atiques, Universitat Polit\`ecnica de Catalunya, 	
	Spain 
	
	{\tt merce.mora@upc.edu}}

\affil[2]{Departamento de Matem\'aticas\\Universidad Aut\'onoma Metropolitana-Iztapalapa, 
	
	M\'exico 
	
	{\tt jtey@xanum.uam.mx}}

\date{}
\maketitle


\begin{abstract}
Let $G=(V,E)$ be a simple graph of size $m$ and $L$ a set of $m$ distinct real numbers. 
An {\em $L$-labeling} of $G$ is a bijection $\phi: E \rightarrow  L$. We say that $\phi$ is an \emph{antimagic $L$-labeling} if the induced \emph{vertex sum} $\phi_+: V \rightarrow  \mathbb {R}$ defined as $\phi_+(u)=\sum_{uv\in E}\phi(uv)$ is injective. Similarly, $\phi$ is a {\em product antimagic $L$-labeling} of $G$ if the induced \emph{vertex product} $\phi_{\circ}: V \rightarrow  \mathbb {R}$ defined as $\phi_{\circ}(u)=\prod_{uv\in E}\phi(uv)$ is injective. A graph $G$ is \emph{antimagic} (resp. \emph{product antimagic}) if it has an antimagic (resp. a product antimagic) $L$-labeling for $L=\{1,2,\dots,m\}$.
Hartsfield and Ringel conjectured
that every simple connected graph distinct from $K_2$ is antimagic, but the conjecture remains widely open.

We prove, among other results, 
that every connected graph of size $m$, $m \geq 3$, 
admits an antimagic $L$-labeling for every arithmetic sequence $L$ of $m$ positive real numbers, if every vertex of degree at least three is a support vertex. As a corollary, we derive that these graphs are antimagic, reinforcing the veracity of the conjecture by Hartsfield and Ringel.
Moreover, these graphs admit also a product antimagic  $L$-labeling provided that the smallest element of $L$ is at least one. The proof is constructive.
\end{abstract}

\maketitle
\section{Introduction}

An  \emph{$L$-labeling} of a graph $G=(V,E)$ of size $m$ is a bijection $\phi: E \mapsto L$, where $L$ a subset of $m$ real numbers.
The \emph{vertex sum} induced by the $L$-labeling $\phi$ at vertex $u$, denoted by $\phi_+(u)$,  is 
the sum of the labels of the edges incident with $u$.
A graph $G$ is {\em antimagic} if there is an $L$-labeling $\phi:E \mapsto L$ for $L=\{1,2,\dots,m\}$ such that $\phi_+(u)\not= \phi_+(v)$, if $u\not=v$.
Hartsfield and Ringel \cite{HR} conjectured in 1990 that every connected graph different from $K_2$ admits an antimagic labeling. 

\begin{conjecture}[\cite{HR}]\label{H-Ringel}
Every connected graph of size at least 2 is antimagic.
\end{conjecture}

This conjecture has received much attention (see \cite{G}), but it is
widely open in general. 
Here we deal with some variations of antimagic labelings, on the one hand by considering different kind of sets of labels and, on the other hand, by considering vertex products instead of vertex sums. 

Let $\mathbb{R}^{+}= (0, \infty)$ and let $G=(V,E)$ be a graph of size $m$.
Recall that the vertex sum induced by an $L$-labeling $\phi$ of $G$ is $\phi_+(u)=\sum_{uv\in E}\phi(uv)$. 
An  $L$-labeling $\phi$ of $G$ is \emph{antimagic} if the induced vertex sum $\phi_+:V\mapsto \mathbb{R}$ is injective. 
We say that $G$ is \emph{universal antimagic} if  $G$ has an antimagic  $L$-labeling for every set $L \subseteq \mathbb{R}^{+}$ of size $m$;
$G$ is \emph{arithmetic antimagic} if   $G$ has an antimagic  $L$-labeling for every arithmetic sequence $L\subseteq \mathbb{R}^{+}$  of length $m$; and
$G$ is \emph{antimagic} if $G$ has an antimagic $L$-labeling for $L=\{1,2,\dots,m\}.$ 

Similar concepts can be defined for vertex products. Let $\mathbb {R}_{\circ}= [1, \infty).$ 
The \emph{vertex product} induced by $\phi$ at vertex $u$ is $\phi_{\circ}(u)=\prod_{uv\in E}\phi(uv)$. 
An  $L$-labeling $\phi$ is \emph{product antimagic}  if the induced vertex product $\phi_{\circ}:V\mapsto \mathbb{R}$ is injective. We say that a graph $G$ of size $m$ is \emph{universal product antimagic} if  $G$ has a product antimagic $L$-labeling for every set $L \subseteq \mathbb {R}_{\circ}$ of size $m$;
$G$ is  \emph{arithmetic product antimagic} if $G$ has a product antimagic $L$-labeling for every arithmetic sequence $L \subseteq \mathbb {R}_{\circ}$ of length $m$; and
$G$ is \emph{product antimagic} if $G$ has a {product antimagic} $L$-labeling for $L=\{1,2,\dots,m\}.$ 

Notice that if a graph $G$ is universal or arithmetic (product) antimagic, then $G$ is (product) antimagic.

Product antimagic labelings were introduced in 2000 by Figueroa-Centeno et al. and the following conjecture was posed.
\begin{conjecture}[\cite{FIM}]\label{ProdAntConj}
Every connected graph of size at least three is product antimagic.
\end{conjecture}

In this paper we focus on a class of graphs with pendant edges, that can also have many vertices of degree two.

Antimagic trees are examples of antimagic graphs with pendant edges. One of the best known results for trees,
is that any tree having at most one vertex of degree two is antimagic (\cite{KLR,LWZ}). Known families of antimagic trees with possibly many vertices of degree two are paths \cite{HR}, subdivision of trees without vertices of degree two \cite{LWZ}, spiders \cite{Shang}, double spiders \cite{CCLP2}, caterpillars \cite{DL1,LMS,LMST}, Fibonacci trees and binomial trees \cite{SSh}, and trees whose even-degree vertices induce a path \cite{LWZ,LMST2}.

Excluding the already mentioned examples of antimagic trees, few is known about antimagicness of graphs with pendant edges or with many vertices of degree two. Examples of antimagic graphs with possibly many vertices of degree two  and without pendant edges are cycles \cite{HR}, graphs that admit a $C_p$-factor, where $p$ is a prime number \cite{HST} and the subdivision of a regular graph \cite{DL2}.
Antimagic graphs with pendant edges and possibly many vertices of degree two are generalized coronas and snowflakes  \cite{DIMP}.

Known product antimagic graphs with pendant edges or with many vertices of degree two are paths, 2-regular graphs \cite{FIM}, the disjoint union of cycles and paths where each path has at least three edges \cite{KLR2}, connected graphs with $n$ vertices and $m$ edges where $m > 4n\ln n$ \cite{KLR2}, caterpillars \cite{WG}. In \cite{P} was given a characterization of all large graphs that are product antimagic. 

The main results of this paper are that connected graphs of size at least three are 
universal antimagic and universal product antimagic if every interior vertex is a support vertex  (Section~\ref{section:examples}), 
and are arithmetic antimagic and arithmetic product antimagic if every vertex of degree at least three is a support vertex (Section~\ref{section:arith ant}).
These results generalize those obtained by Lozano et al. \cite{LMST}  
and Wang and Gao \cite{WG}, 
 who proved that caterpillars are antimagic and product antimagic, respectively, and reinforces the veracity of Conjectures \ref{H-Ringel} and \ref{ProdAntConj}.

We finish this section by setting some terminology.
In what follows, all graphs considered are simple. The set of vertices and the set of edges of a graph $G$ are denoted by $V(G)$ and $E(G)$, respectively. The \emph{degree} of $v \in V(G),$ denoted by $\deg_G(v),$ is the number of edges incident with $v$. We write simply $V$, $E$ and $\deg(v)$ when the underlying graph is understood. 
A \emph{leaf} is a vertex of degree one
and a \emph{pendant edge} is an edge incident with a leaf.
The number of leaves in $G$ will be denoted by $\ell(G)$. 
An \emph{interior} vertex is a vertex of degree greater than one. 
The set of interior vertices of $G$ is denoted by $V_I(G)$.
A \emph{support} vertex is 
a vertex with at least one pendant edge incident with it.
The set of support vertices in $G$ is denoted by $V_s(G)$. 
For undefined terms we refer the reader to \cite{CHL, West}.

Given a graph $G$, a {\em leafy graph} of $G$  is a graph obtained by adding a non-negative number of pendant edges to each interior vertex of $G$. The set of leafy graphs of $G$ will be denoted by $\mathcal{H}(G).$ Notice that $G \in \mathcal{H}(G).$

The following definition will be useful later.
\begin{definition}
When constructing a labeling $\phi$ of $G$,
we say that a vertex $v \in V(G)$ is {\em almost saturated} if all but one edges incident with $v$ have been labeled; $\phi^*_{+}(v)$ (resp. $\phi^*_{\circ}(v)$) denotes the sum (resp. product) of the labels of the already labeled edges incident with $v$, and $V^*(G)$ denotes the currently set of vertices of degree at least three almost saturated  in $G$.
A vertex $v \in V$ is {\em saturated} when all the edges incident with $v$ have been labeled.
\end{definition}

\section{Universal antimagic and product antimagic graphs}\label{section:examples}

We use the notion of weighted antimagic graph introduced in \cite{MZ1} and its generalization by considering vertex products,  to prove the main results of this section.

Let $G = (V,E)$ be a graph of size $m$.
We say that $G$ is {\em weighted universal antimagic} if for any vertex weight function $w: V \rightarrow \mathbb {R}^+$ and every set 
$L \subseteq \mathbb{R}^{+}$ of size $m$ there is a bijection $\phi: E \rightarrow L$ such that $w(u)+\phi_+(u) \neq w(v)+\phi_+(v)$, for any two distinct vertices $u$ and $v.$ 
Similarly, $G$ is {\em weighted universal product antimagic} if for any vertex weight function $w: V \rightarrow \mathbb {R}_{\circ}$ and every set $L \subseteq \mathbb {R}_{\circ}$ of size $m,$  there is a bijection $\phi: E \rightarrow L$ such that $w(u)\phi_{\circ}(u) \neq w(v)\phi_{\circ}(v)$, for any two distinct vertices $u$ and $v.$ 

Matamala and Zamora prove the following results in  \cite{MZ1}.

\begin{proposition}\label{P and C U-A} \cite{MZ1}
Paths and cycles of size at least three are universal antimagic.
\end{proposition}

\begin{proposition}\label{W U Spanning} \cite{MZ1}.
  Let $H$ be a spanning subgraph of a graph $G.$ If $H$ is weighted universal antimagic, then $G$ is weighted universal antimagic.
\end{proposition}

The following two results are obtained straightforward by replacing sums with products in the 
procedures given in \cite{MZ1} to prove Propositions~\ref{P and C U-A}  and~\ref{W U Spanning}.

\begin{proposition}\label{P and C U-A-prod}
Paths and cycles of size at least three are universal product antimagic.
\end{proposition}
\begin{proposition}\label{W$ U Spanning}
Let $H$ be a spanning subgraph of a graph $G.$ If $H$ is weighted universal product antimagic, then $G$ is weighted universal product antimagic.
\end{proposition}

Next, we give some results for leafy graphs by means of weighted universal (product) antimagic labelings.

\begin{proposition}\label{w-universal}
If $G$ is a weighted universal (product) antimagic graph, then any leafy graph of $G$ is universal (product) antimagic.
\end{proposition}

\begin{proof}
Let $\widetilde G$ be a leafy graph of $G,$ and let $H$ be the graph induced by the set of pendant edges $E(\widetilde G)\setminus E(G)$. Suppose that $\vert E(\widetilde G) \vert=m$ 
and $\vert E(H) \vert = \vert E(\widetilde G)\setminus E(G) \vert= h$.

We prove first that $\widetilde G$ is universal antimagic.
Let $L=(l_1, \dots, l_m)$ be a strictly increasing sequence in $\mathbb{R}^+.$ 
Consider a bijection  $\phi_H: E(H) \rightarrow \{l_1, \dots, l_h\}$ and define a weight function in $G$, 
$w: V(G) \rightarrow \mathbb {R}^+$, as follows
\begin{equation*}\label{eq:weight}
	w(x) =
	\begin{cases*}
	 \displaystyle\sum_{xy \in E(H)} \phi_H (xy), &
	\mbox{if $x \in V(H),$}\\
	\hspace{1cm}0, & \mbox{if $x \notin V(H)$.}
	\end{cases*}
	\end{equation*}

Since $G$ is weighted universal antimagic, there exists an antimagic labeling of $G$, $\theta : E(G) \rightarrow \{l_{h+1}, \dots, l_m\}$, such that 
\begin{equation}\label{eq:theta}
w(u)+\theta_+(u)\neq w(v)+\theta_+ (v)    
\end{equation}
for any two distinct vertices $u$ and $v$ of $G$. Now we extend the labeling $\phi_H$ to a labeling $\phi$ of $\widetilde G$ by defining $\phi(e)=\theta(e)$, for every $e \in E(G)$. 
We claim that $\phi$ is an antimagic $L$-labeling of $\widetilde G.$
Indeed, by construction, $\phi:E(\widetilde G)\mapsto L$ is a bijection.
Notice that if $u\in V(G)$, then 
\begin{equation}\label{eq:phi1}
\phi_+(u)=\sum_{uu'\in E(\widetilde G)} \phi(uu')
=\sum_{uu'\in E(H)} \phi(uu')+ \sum_{uu'\in E(G)} \phi(uu')=w(u)+ \theta_+(u)> l_h    
\end{equation}
and if  $u\notin V(G)$, then $u$ is a leaf incident with a pendant edge $e_u$ of $H$ and 
\begin{equation}\label{eq:phi2}
\phi_+ (u)=\phi (e)=\phi_H(e_u)\le l_h.
\end{equation}

Now let $x$ and $y$ be two distinct vertices of $\widetilde G$.
If $x,y \notin V(G)$, then $\phi_+(x)=\phi_H(e_x)\not= \phi_H (e_y)=\phi_+(y)$, since $\phi_H$ is a bijection.
If  $x \notin V(G)$ and $y \in V(G),$ then 
$\phi_{+}(x) \leq l_h < \phi_{+}(y),$ by Inequalities (2) and (3).
If  $x,y \in V(G),$ then 
\begin{equation}\label{eq:theta2}
\phi_+(x)=w(x)+ \theta_+(x)\not= w(y) + \theta_+ (y)=\phi_+ (y)   
\end{equation} 
because $\theta$ is a weighted antimagic labeling for $w$.
Hence, $\phi$ is an antimagic $L$-labeling of $\widetilde G$ and, consequently, $\widetilde G$ is universal antimagic. 

To prove that $\widetilde G$ is universal product antimagic we proceed similarly. In that case, suppose that $L=(l_1, \dots, l_m)$ is a strictly increasing sequence in $\mathbb {R}_{\circ}$. 
Consider as before a bijection  $\phi_H: E(H) \rightarrow \{l_1, \dots, l_h\}$ and define a
weight function $w: V(G) \rightarrow \mathbb {R}_{\circ}$ in $G$ 
as follows
\begin{equation*}\label{eq:prod}
	w(x) =
	\begin{cases*}
	 \displaystyle\prod_{xy \in E(H)} \phi(xy), &
	\mbox{if $x \in V(H),$}\\
	\hspace{1cm}1, & \mbox{if $x \notin V(H)$.}
	\end{cases*}
	\end{equation*}
Then, since Expressions (\ref{eq:theta}), (\ref{eq:phi1}), (\ref{eq:phi2}) and (\ref{eq:theta2}) hold also if we replace $\phi_+$, $\theta_+$ and sums with $\phi_{\circ}$, $\theta_{\circ}$ and products, respectively, we get that $\phi$ is a product antimagic $L$-labeling of $\widetilde G$. 
Therefore, $\widetilde G$ is universal product antimagic. 
\end{proof}

As a corollary of Proposition \ref{w-universal} and the following known result, we obtain a family of universal antimagic graphs.

\begin{theorem}[\cite{MZ1}]\label{Bipartite leafy}
Let $m, n \geq 3$.
If $K_{m,n}$ is a spanning subgraph of a graph $G$, then
 $G$ is weighted universal antimagic.
\end{theorem} 

\begin{corollary}
Let $m, n \geq 3$. If $K_{m,n}$ is a spanning subgraph of a graph $G$, then any leafy graph of $G$ is universal antimagic.
\end{corollary}

Next proposition provides another family of universal (product) antimagic graphs. 

\begin{proposition}\label{Full leafy}
Let $G$ be a connected graph of size $m \geq 3.$ If every interior vertex of $G$ is a support vertex, then $G$ is universal antimagic and universal product antimagic. 
\end{proposition}

\begin{proof}
Let $(l_1, l_2, \cdots , l_m)$ be a strictly increasing  sequence in $\mathbb {R}^{+}$. Suppose $G$ has $n_{\ell}$ leaves and $n_s$ support vertices. Let $\phi $ be a labeling that arbitrarily assigns labels in $[l_{n_s+1},l_{n_{\ell}}]$ to the all but one pendant edges incident with each vertex in $V_I(G)$ and arbitrarily assign a label in the (possibly empty) set $[l_{n_{\ell}+1},l_m]$ to non-pendant edges. 
Let $(v_1,v_2,...,v_{n_s})$ be the order of the elements of $V_I(G)$ so that {$\phi^*_{+}(v_i) \leq \phi^*_{+}(v_{i+1})$}. Let $p_i$ be the unlabeled pendant edge incident with $v_i$. For $1 \leq i \leq n_s$, label  $p_i$ with $l_i$. It is clear that the resulting labeling of $G$ is antimagic, since the vertex sums induced by the labeling at the leaves are  $l_1,\dots,l_{n_{\ell}}$
and, by construction, the sums at interior vertices are greater than $l_{n_{\ell}}$ and pairwise distinct.

Now let $(l_1, l_2, \cdots , l_m )$ be a strictly increasing sequence in $\mathbb {R}_{\circ}$. If we proceed as before but considering the order $(v_1,v_2,\dots ,v_{n_s})$ of the vertices 
in $V_I(G)$ so that {$\phi^*_{\circ}(v_i) \leq \phi^*_{\circ}(v_{i+1})$}, we obtain a product antimagic labeling. 

Hence, $G$ is universal antimagic and universal product antimagic.
\end{proof}

The {\em corona} $G_1 \circ G_2$ of two graphs $G_1$ and $G_2$ is the graph obtained by taking one copy of $G_1$  and $\vert V(G_1)\vert$ copies of $G_2$, and then joining the ith vertex of $G_1$ to every vertex of the ith copy of $G_2$.
It was proved in \cite{KLR2} that if $G_1$ is a graph without isolated vertices and $G_2$ is a regular graph, then $G_1 \circ G_2$ is product antimagic. 
Note that taking $G_2$ as 
an empty graph
is a particular case of Proposition \ref{Full leafy}. 

\section{Arithmetic antimagic and product antimagic graphs}\label{section:arith ant}

In this section we show that the general idea used in \cite{LMST} to prove that caterpillars are antimagic can be adapted to produce an arithmetic (product) antimagic labeling of any graph 
such that every vertex of degree at least three is a support vertex (notice that caterpillars satisfy this condition). 

\begin{theorem}\label{Antimagic no bald}
Let $G$ be a connected graph of size at least three. If every vertex of degree at least three is a support vertex, then $G$ is arithmetic antimagic and arithmetic product antimagic.
\end{theorem}

To prove this theorem, we provide a labeling algorithm in Section~\ref{sec:algo} and show its correctness in Section~\ref{sec:correct}. To make the procedure clearer, some examples of the labeling produced by the algorithm are given in Section~\ref{sec:examples}. In Section~\ref{sec:concluding}, some concluding remarks are given.

Since the procedure is the same for antimagic labelings and for product antimagic labelings, we simply write $\oplus$-antimagic labeling, $\phi_{\oplus}(u)$ and $\mathbb{R}_{\oplus}$,
where $\oplus \in \{+,\circ\}$. So, we have to read antimagic labeling, $\phi_{+}(u)$ and $\mathbb{R}^+$, respectively,
when $\oplus =+$, and product antimagic labeling, $\phi_{\circ}(u)$ and $\mathbb{R}_{\circ}$, respectively,
when  $\oplus =\circ$.

\subsection{The labeling algorithm}\label{sec:algo}

We introduce first some terminology. The set of 
vertices of degree at least three is denoted by $V_3(G)$; 
and the set of support vertices such that 
all but one of their neighbors are leaves, by $V_s'(G)$. 
The set of vertices in $V_s'(G)$ that have degree at least three is denoted $V_{s,3}'(G)$.  
We will simply write $V_3,$ $V_s',$ and $V_{s,3}'$ if the graph $G$ is clear from the context.
A {\em star} is a tree of order at least three with exactly one interior vertex called the {\em center} of the star.

\begin{definition}
Given a graph $G$, a {\em pruned} graph of $G$ is a graph that results by 
deleting from each vertex $v\in V_s'$ all but one leaves adjacent to $v$ and deleting all leaves adjacent to any other support vertex not in $V_s'$. Exceptionally, the pruned graph of $K_2$ is the same graph $K_2$ and if $G$ is a star, then a pruned graph of $G$ is a path of order three. 
\end{definition}
\begin{figure}[ht]
    \centering    \includegraphics[scale=.7]{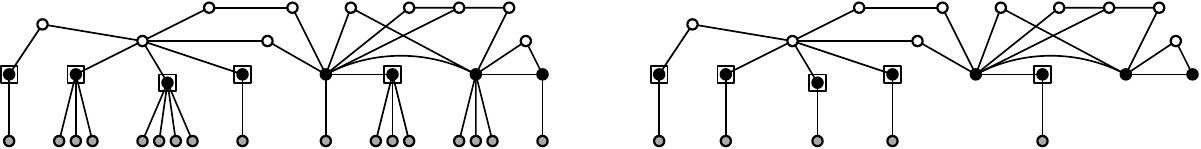}    \caption{Left, a graph $G$. In black, the vertices in $V_s$. Squared black are the vertices in $V_s'$. Right, the pruned graph of $G$: all leaves have been deleted, except one leaf hanging from vertices in $V_s'$. Notice that  Theorem~\ref{Antimagic no bald} does not apply to $G$ because some vertices of degree at least 3 are not support vertices.}
\label{fig:pruned}
\end{figure}

In Figure~\ref{fig:pruned}, an example of a pruned graph is depicted.
Notice that all pruned graphs of a graph $G$ are isomorphic. Moreover, a graph and its pruned graph have exactly the same set of interior vertices.

\begin{definition}\label{GPC}
Given a tree $T$ and $v \in V(T)$, a \em{Good Path Decomposition (GPD for short) of $T$ centered at $v$} is a collection of paths $\{P_1, \dots, P_r\}$ satisfying the following conditions:

\begin{itemize}
\item $P_1, \dots, P_r$ is a decomposition of $T$.
\item $v$ is an end-vertex of $P_1.$
\item  For $ 1 \leq i \leq r$, one end-vertex of $P_i$ is a leaf of $T$. 
\item For $ 2 \leq j \leq r$, there exists $i < j$ such that one end-vertex of $P_j$ is a vertex of $P_i.$
\end{itemize}
\end{definition}

Notice that  for any given  tree $T$ and $v \in V(T)$, a GPD of $T$ centered at $v$ can be constructed using Depth First Search Algorithm (DFS-Algorithm for short). 
In Figure~\ref{fig:GPD}, an example of a GPD is given.

\begin{figure}[ht]
    \centering    \includegraphics[scale=.7]{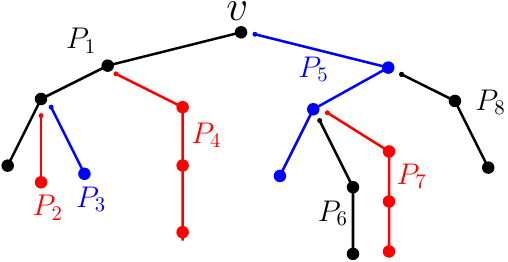}    \caption{A GPD of a tree $T$ into 8 paths centered at $v$ obtained by applying DFS Algorithm.
    Each edge of $T$ belongs to exactly one of the 8 paths.}
\label{fig:GPD}
\end{figure}

Let $G$ be a connected graph of size $m \geq 3$ 
such that $V_3\subseteq V_s$, $|V_I|\ge 2$ and $|V_3|\ge 1$. Let $L=(l_1,l_2,\dots , l_m)\subseteq \mathbb {R}$ be an arithmetic sequence.

We proceed to describe a four-step labeling procedure that provides an $L$-labeling $\phi:E(G)\mapsto L$ that is antimagic, if $L\subseteq \mathbb{R}^+$, and product antimagic, if $L\subseteq \mathbb{R}_{\circ}$. For the sake of clarity we also include a pseudocode.
\medskip

{\bf STEP 1.}
{\sc Labeling almost all pendant edges of $G$.} 
\medskip

For each vertex $v\in V_3$, if $v\in V_{s,3}'$, then we label all but two pendant edges incident with $v$ using the smallest unused labels; if $v\notin V_{s,3}'$, then we label all but one pendant edges incident with $v$ using the smallest unused labels (see Algorithm 1, lines 1--10 and Figure~\ref{fig:examplesteps}(a,b)).
\medskip

{\bf STEP 2.}
{\sc Labeling almost all pendant edges of $G_1$.}
\medskip

Let $\mathcal{P}(G)$ be the pruned graph of $G$ containing no labeled edge.  Consider a partition $\{ E_1, E_2\}$ of the edges of $\mathcal{P}(G)$, such that the subgraph $G_1$ induced by $E_1$ is a forest and the subgraph $G_2$ induced by $E_2$ is an even graph (see Figure~\ref{fig:examplesteps}(c,d)), where $E_1$ or $E_2$ (but no both) can be the empty set. Notice that this partition is always possible, since we can successively remove from the graph a set of edges of a cycle until we get a forest, and the removed edges induce an even graph.   
\medskip

Now,
we label almost all pendant edges of $G_1$ using the smallest unused labels as follows. Indeed, let $C$ be a component of $G_1$.

If $C$ is a star, then label all pendant edges of $C$ except two.
In any other case, for each vertex $v$ in $V_3(C)$,
label all but one pendant edges incident with $v$, if $v\in V_{s,3}'(C)$; and 
 label all pendant edges incident with $v$, if $v\notin V_{s,3}'(C)$  
(see Algorithm 1, lines 12--29).
\medskip
\begin{figure}
    \centering    \includegraphics[width=0.95\textwidth]{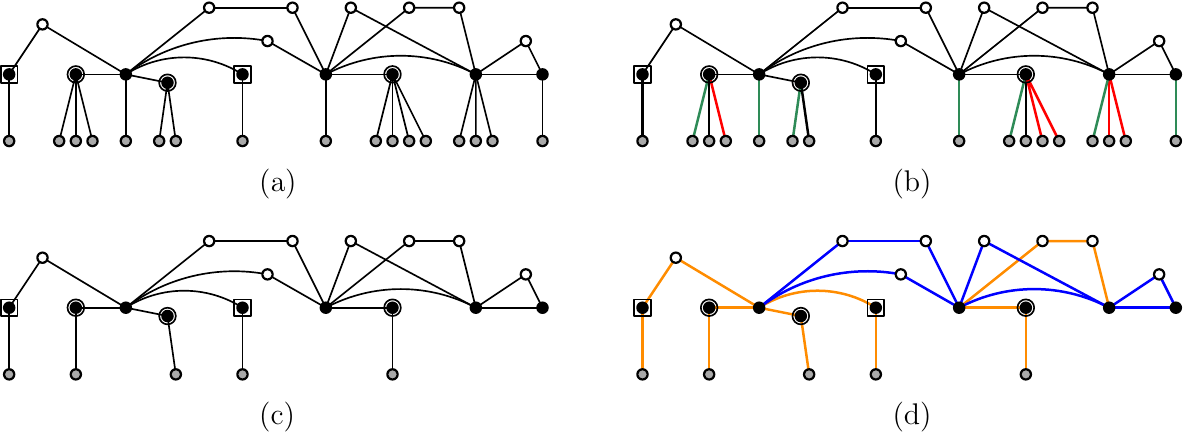}    \caption{(a) A graph $G$ such that $V_3\subseteq V_s$. Vertices in $V_s$ are in black. Squared and circled black vertices are those in $V_s'$ and circled vertices are the vertices in $V_{s,3}'$. 
    (b) Red edges are labeled in Step 1. Green edges will be labelled in Step 3 or 4.
    (c) The pruned graph $\mathcal{P}(G)$.
    (d) $G_1$ and $G_2$ are the subgraphs of $\mathcal{P}(G)$ induced by orange edges and blue edges, respectively. In this case, $G_1$ is a forest with two components and $G_2$ is a connected even graph. The edges of $G_1$ and $G_2$ are labeled in Step 3. No edge is labeled in Step 2 in this example.}
\label{fig:examplesteps}
\end{figure}

{\bf STEP 3.} 
{\sc Labeling all $S$-edges and possibly some pendant edges of $G$.}
\medskip

Let $(C_{1},C_{2},\dots,C_{t_1})$ and $(C'_{1},C'_{2},\dots,C'_{t_2})$ be the components of $G_1$ and $G_2$, respectively. Let $\mathcal{P}(C_i)$
be a pruned graph of $C_i$ containing no labeled edge.
For $1 \leq i \leq t_1,$ 
let $(P_{i,1},P_{i,2},\dots,P_{i,n_i})$ be a 
GPD of $\mathcal{P}(C_i)$ centered at a vertex $v_i\in V(\mathcal{P}(C_i))$ chosen as follows.
If $E(G_2)\not= \emptyset$, choose $v_i$ be such that 
$v_i\in V(\mathcal{P}(C_i))\cap V(G_2)$. 
If $E(G_2)=\emptyset$, then $t_1=1$ and we choose 
$v_1 \in V_3(G).$ Notice that this is possible, since we assumed $G$ is not a path 
(see Algorithm 1, lines 32--34).

We now partition the set $L$ into two sets, 
the set $L_s$ of \emph{smaller} labels and the set $L_b$ of \emph{bigger} labels,
where the the set of bigger labels will be used to label approximately half of the edges of each $P_{i,j}$ and each $C_i$.

Concretely, let
$$m_0= \sum_{i=1}^{t_1} \sum_{j=1}^{n_i} \bigg \lfloor \frac {\vert E(P_{i,j}) \vert }{2} \bigg \rfloor +  \sum_{i=1}^{t_2} \bigg \lceil  \frac {\vert E(C'_i) \vert }{2} \bigg \rceil$$ 
and
$$m_b= \left\{ \begin{array}{cc} m_0 +1 &\hbox{ if }G=G_1\hbox{ and } |E(P_{1,1})|\hbox{ is odd,  }\\ m_0& \hbox{ otherwise. }\end{array} \right.  $$

Then,  $L_b$ is the set containing  the $m_b$ biggest labels, that is, $L_b=(l_{i_b}, l_{i_b+1}, \dots,l_m)$, where $i_b=m-m_b+1$, and  $L_s=L \setminus L_b$.
Notice that $L_s$ contains the labels already used in Steps 1 and 2.

Now consider the sequence 
$$S=(P_{1,1},P_{1,2},\dots,P_{1,n_1},P_{2,1},\dots,P_{2,n_2},\dots,P_{t_1,1},\dots,P_{t_1,n_{t_1}}, C'_{1},C'_{2},\dots,C'_{t_2}).$$
An edge in the union of all subgraphs in $S$ is called an {\em $S$-edge}.
Next, following the order of the elements in $S$ we label all $S$-edges and, possibly, some pendant edges of $G$. 
We begin by ordering the edges of $H\in S$ as follows.
If $H$ is one of the paths $P_{1,1},\dots,P_{1,n_1},\dots,P_{t_1,1},\dots,P_{t_1,n_{t_1}}$, then sort the edges as they appear when going through the path beginning at a leaf in $\mathcal{P}(C_i)$ that is not the center of the GPD. 
If $H\in \{C_1',\dots,C_{t_2}'\}$, then $H$ has an Eulerian circuit. In this case, sort the edges as they appear in an Eulerian circuit but beginning and ending at a vertex $v$ of degree at least 3 in $G$. Notice that this is always possible by the assumptions made on $G$.

For $H\in S$, let $\hat{H}=(e_1,\dots ,e_n)$ be the ordered sequence of edges of $H$ described above.
Then, we label successively the edges $e_1,\dots ,e_n$ with alternate smallest unused labels from $L_s$ and $L_b$, starting with a label in $L_s$ except in the following two cases  (see Algorithm 2). 

\begin{enumerate}

\item[$\bullet$] If $G$ is a tree, $H=P_{1,1}$ and $H$ has odd size,
then we begin with a label from $L_b$ instead of $L_s$, except in the following case. If we are constructing a product antimagic $L$-labeling, and $1\in L$ and the label $1$ has not been assigned in Step 1, then begin by assigning the labels 1 and {the second smallest label} to the edges $e_1$ and $e_2,$ respectively, and continue with alternate smallest unused labels from $L_b$ and $L_s$ (see Example 2 below).
   
    \item[$\bullet$] If $H\in \{C_1',\dots,C_{t_2}'\}$ and $H$ has of odd size, then begin with a label from $L_b$ 
    and continue with alternate smallest unused labels from $L_s$ and $L_b$.
\end{enumerate}

When assigning a label from $L_s$ to an edge {$e$ in $\hat{H}$} we ask first if {\bf Condition $Q(w)$} holds at some vertex $w \in V^*(G)$.
\medskip

{\bf Condition $Q(w)$. }$ 
\phi_{\oplus}^*(w) \leq\ell _b$, where
$\phi_{\oplus}^*(w)=min \{\phi_{\oplus}^*(v): v \in V^* \}$ and $\ell_b$ is the current biggest used label in $L_b$.
\medskip

While {\bf Condition $Q(w)$} holds, then we assign first the smallest unused label to the unlabeled pendant edge incident with $w$ (see Algorithm 2, lines 13--16),  and continue by assigning the next smallest label in $L_s$ to the edge $e$ in $\hat{H}$.

\medskip
{\bf STEP 4.}
{\sc Labeling the remaining pendant edges.}
\medskip

Finally,  we properly label the final set of unlabeled pendant edges with the remaining unused labels of $L_s.$
Notice that, at this step, 
each unlabeled pendant edge is incident with an almost saturated interior vertex.
Sort the vertices in $V^*(G)$ as $(v_1,\dots,v_t)$ so that 
$\phi_{\oplus}^*(v_i) \le \phi_{\oplus}^*(v_{i+1})$, if $1\le i < t$,
and assign the smallest unused label to the unlabeled pendant edge incident with $v_i$.

\begin{algorithm}
	\caption{Arithmetic $\oplus$-antimagic labeling of a connected graph such that $V_3\subseteq V_s$}\label{algo}
	\ \\ \
	\hglue 5pt \textbf{Input:} A connected graph $G$  of size $m\ge 3$ such that $V_3\subseteq V_s$, $|V_I|\ge 2$ and $|V_3|\ge 1$\\
	\hglue 45pt An arithmetic sequence $L=(l_1,l_2,\dots,l_m) \subseteq \mathbb {R}_\oplus$\\
	\hglue 5pt \textbf{Output:} An ${\oplus}$-antimagic $L$-labeling $\phi$ of $G$
	\begin{algorithmic}[1]
		\Statex
		\Statex \hglue -4.7mm {\sc{\bf STEP 1:} Labeling almost all pendant edges of $G$}
		\vspace{0.2cm}
		\For{each vertex {$v \in V_{s,3}'(G)$}}  
		     \For{each pendant edge $h$ incident with $v,$ except two}
		      \State $\phi(h) \gets$ the smallest unused label
		   \EndFor	
		\EndFor
		\For{each vertex {$v\in V_3(G)\setminus V_{s,3}'(G)$}}
		   \For{each pendant edge $h$ incident with $v,$ except one}
		      \State $\phi(h) \gets$ the smallest unused label
		   \EndFor		
		\EndFor
		\Statex
		
		\Statex \hglue -6mm {\sc{\bf STEP 2:} Labeling almost all pendant edges of $G_1$}
		\vspace{0.2cm}
		
		\State $G^* \gets \mathcal P(G),$ in such a way that $G^*$ has no labeled edges
		\Statex Let $\{ E_1, E_2\}$ be a partition of the edges of $G^*$, such that the subgraph $G_1$ induced by $E_1$ is a forest and the subgraph $G_2$ induced by $E_2$ is an even graph
		\Statex Let $(C_{1},C_{2},\dots,C_{t_1})$ be an order of the components of $G_1$
		\For{$i=1$ to $t_1$}
		  \If {$C_i $ is a star}
		   \For{each pendant edge $h$ incident with the center of $C_i,$ except two} 
		     \State $\phi(h) \gets$ the smallest unused label
		   \EndFor
		   \Else
		       \For{each $v \in V_{s,3}'(C_i)}$
		     \For{each pendant edge $h$ in $C_i$ incident with $v,$ except one}
		      \State $\phi(h) \gets$ the smallest unused label
		   \EndFor	
		 \EndFor
		  \For{each vertex {$v \in {V_3(C_i) \setminus V_{s,3}'(C_i)}$}}
		    \For{each pendant edge $h$ in $C_i$ incident with $v$}
		      \State $\phi(h) \gets$ the smallest unused label
		   \EndFor	
		 \EndFor 		 
		\EndIf 
		\EndFor   
	    
		\Statex
		 \algstore{myalgo}
	\end{algorithmic}
\end{algorithm}
\begin{algorithm}
	\ContinuedFloat
          \caption{Arithmetic $\oplus$-antimagic labeling of a connected graph
         such that $V_3\subseteq V_s$ (continued)}
          \ \\ \
		\begin{algorithmic}
			\algrestore{myalgo}
		\Statex \hglue -6mm {\sc{\bf STEP 3:} Labeling all $S$-edges and possibly some pendant edges of $G$}
		\vspace{0.2cm}

		\State $L_b \gets (l_{i_b}, \dots,l_m)$  \Comment{{\color{blue}The index $i_b$ can be previously determined}}
		\State $\ell_b \gets 0$ \Comment{{\color{blue}$\ell_b$ is the last used label in $L_b$}}
		\For{$i=1$ to $t_1$}  
		    \State $C \gets \mathcal P(C_i),$ in such a way that $C$ has no labeled edges
		    \State Find a GPD $(P_1,P_2,\dots,P_q)$ of $C$ centered at $v \in V(C).$ 
      
      If $G_2 = \emptyset,$ then $v \in V_3(C),$ otherwise we choose $v\in V(C)\cap V(G_2)$  		   
		  \For{$j=1$ to $q$} 
		    \State Find an Eulerian circuit $\hat{H}=(e_1,e_2,\dots,e_n)$ of $P_j$ such that one end-vertex in 
      
      \phantom{xxx}$e_1$  is a leaf of $C,$ and also it is not the center of the corresponding GPD. 
		    \State Apply {\bf Algorithm 2} 
      
      \Comment{{\color{blue}Label the edges in $\hat{H}$, and some pendant edges of $G$ (if needed)}}
		  \EndFor
		\EndFor
  \vspace{1mm}
  
		\Statex \phantom{x}Let $(C'_{1},C'_{2},\dots,C'_{t_2})$ be an order of the components of $G_2$ 
   \vspace{1mm}
   
		\For{$i=1$ to $t_2$}
		   \State Find an Eulerian circuit $\hat{H}=(e_1,e_2,\dots,e_n)$ of $C'_i$ such that $e_1 \cap e_n \in  V_3(G)$ 
		   \State Apply {\bf Algorithm 2} 
     
     \Comment{{\color{blue}Label the edges in $\hat{H}$, and some pendant edges of $G$ (if needed)}}
		\EndFor
		
			\Statex
		
		\Statex \hglue -6mm {\sc{\bf STEP 4:} Labeling the remaining pendant edges}
		\vspace{0.2cm}
		\State Sort the vertices in $V^*(G)$ as $(v_1,\dots,v_t)$ such that for all $i < t,$
$\phi_{\oplus}^*(v_i) \le \phi_{\oplus}^*(v_{i+1})$ 		
\Statex \hspace{0.1cm} Let $h_i$ be the unlabeled pendant edge incident with $v_i$
		\For{$i=1$ to $t$}
		   \State $\phi(h_i) \gets$ the smallest unused label 
		\EndFor
		\end{algorithmic}
\end{algorithm}

\begin{algorithm}
	\caption{Labeling the edges of $\hat{H}$ and some pendant edges of $G$ (if needed)}\label{algo2}
	\ \\ \
	\hglue 5pt \textbf{Input:} {A trail} $\hat{H}=(e_1,e_2,\dots,e_n)$\\
	\hglue 5pt \textbf{Output:} A labeling of the edges in $\hat{H}$ and some pendant edges of $G$ {(if needed)}
	\begin{algorithmic}[1]
		\Statex
		 \State first index $\gets 1$
		 \State $r \gets 0$
		\If{{(($G$ is a tree and $\hat{H}$ is a trail of the first path in the GPD of $G$) or ($\hat{H}$ is a circuit))} 
  		
			and $n \equiv 1 \pmod 2$}
		   \State $r \gets 1$
		   \State $l \gets$ the smallest unused label
		   \If {$\hat{H}$ is a trail of a path and $ \oplus = \circ$ and $l=1$} 		   	      
		          \State $\phi(e_1) \gets 1$
		          \State first index $\gets 2$
		   \EndIf 
	        \EndIf     
		\For {$i=$ first index to  $n$}
		\If {$i+r  \equiv 1 \pmod 2$}
	   	  \While {$\ell_b \geq \phi_{\oplus}^*(w)=min \{\phi_{\oplus}^*(v): v \in V^*(G) \}$}
		    \Statex \hspace{1.6cm} Let $h_{w}$ be the unlabeled pendant edge incident with $w$
		    \State $\phi(h_w) \gets$ the smallest unused label 
		  \EndWhile
		  \State $\phi(e_i) \gets$ the smallest unused label 
	        \Else
	          \State $\ell_b \gets$ the smallest unused label in $L_b$
		  \State $\phi(e_i) \gets \ell_b$ 
		\EndIf
		\EndFor   
		
	\end{algorithmic}
\end{algorithm}
\clearpage

\subsection{Examples}\label{sec:examples}

{\bf Example 1.}  In Figure \ref{graphex},
a graph $G$ such that $V_3\subseteq V_s$ is depicted. 
Observe that no pendant edge is labeled in Step 1.
The subgraph induced by dashed and black edges is a pruned graph $\mathcal P(G)=(G_1, G_2)$ of $G$ obtained in Step 2, where dashed edges belong to the forest $G_1$ and  black edges belong to the even graph $G_2$. 

\begin{figure}[H]
	\centering
	\includegraphics[scale=.7]{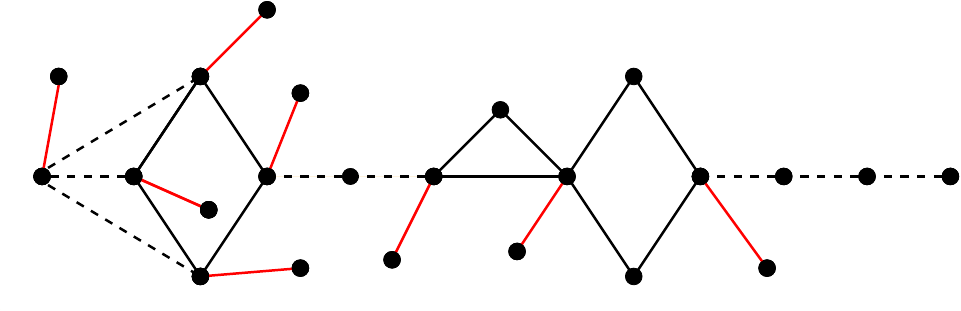} 
	\caption{A graph such that $V_3\subseteq V_s$.}
	\label{graphex}
\end{figure}

The rest of the labeling process is showed in Figure \ref{antlabgraph}. 
Label 1 is assigned in Step 2, because the corresponding edge belongs to a star of $G_1$ (see lines 13--16 in Algorithm 1). The labels of the $S$-edges, colored in blue and green, are assigned in Step 3 by using Algorithm 2. The  labels of red pendant edges are assigned in Step 4. The blue-green subgraph $G_1$ has three components (each of them is a tree); and each component has a GPD centered at the circled vertices $v_1, v_2$ and $v_3,$ respectively. The blue-green subgraph $G_2$ has two components, say $C'_1$ and $C'_2$ (each of them is an even graph). The  Eulerian circuit of $C'_1$ begins (and ends) at $v_1$ and the Eulerian circuit of $C'_2$ begins at $v_3.$ Notice that the dark-red pendant edge with label 7 has been labeled in Step 3 (because {\bf Condition $Q(w)$} holds just before assigning label 8 to a green $S$-edge).

\begin{figure}[H]
	\centering
	\includegraphics[scale=.7]{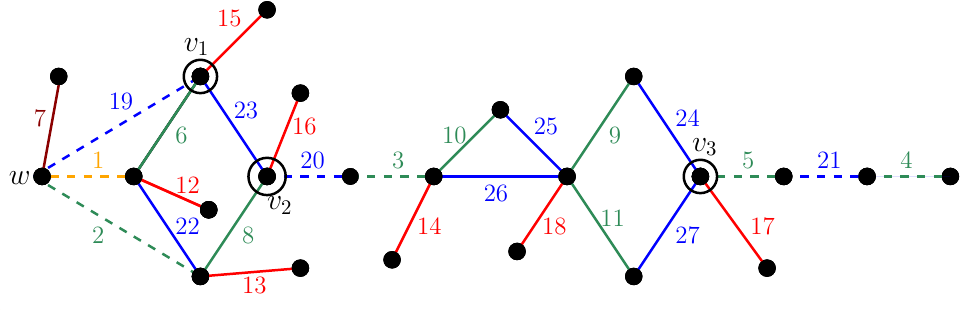} 
	\caption{An antimagic labeling of a graph such that $V_3\subseteq V_s$.}
	\label{antlabgraph}
\end{figure}

{\bf Example 2.} In Figure \ref{antlabtree} we have a product antimagic labeling of a tree $T.$ The circled vertex $v_1$ is the center of the given GPD of $T.$ The $S$-edges, colored in blue and green, are labeled in Step 3, and the red pendant edges are labeled in Step 4. 
Notice that this is an example of the only exception that can appear when applying the algorithm to produce a product antimagic labeling  (see Algorithm 2, lines 6--9).

\begin{figure}[H]
	\centering
	\includegraphics[scale=.7]{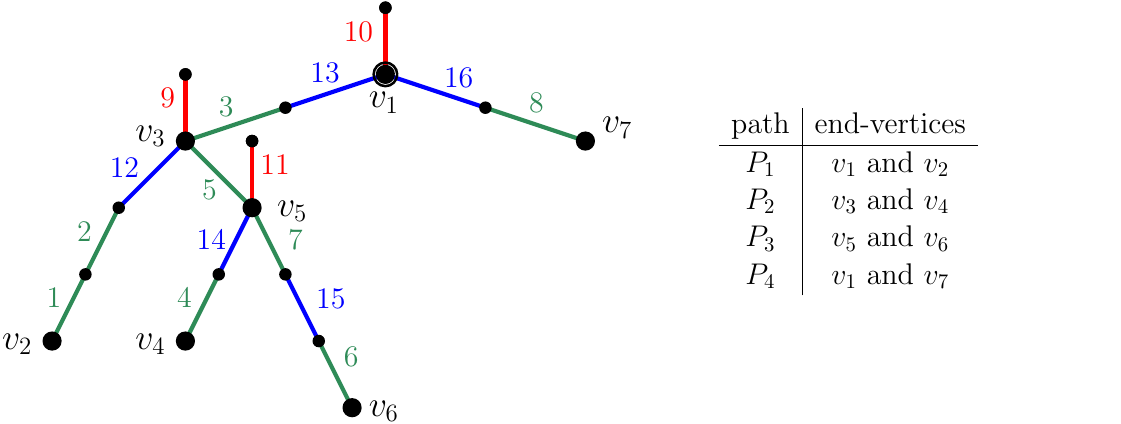}
	\caption{A product antimagic $L$-labeling of a tree $T$ of size 16 
 such that $V_3\subseteq V_s$. The GPD centered at $v_1$ obtained by applying DFS to $\mathcal{P}(T)$ is $(P_1,P_2,P_3,P_4)$.}
	\label{antlabtree}
\end{figure}

\subsection{Proof of Theorem~\ref{Antimagic no bald}}\label{sec:correct}

To prove Theorem~\ref{Antimagic no bald} it suffices to show that 
Algorithm 1 provides an antimagic $L$-labeling, if $L\subseteq \mathbb{R}^+$, and 
a product antimagic $L$-labeling, if $L\subseteq \mathbb{R}_\circ$.

Let $G$ be a connected graph 
of size $m$, $m \geq 3$, such that $V_3\subseteq V_s$. Since stars are trivially universal antimagic and universal product antimagic, we assume that $G$ has at least two internal vertices. Also, by Propositions \ref{P and C U-A} and \ref{P and C U-A-prod}, paths and cycles are universal $\oplus$-antimagic, so we may assume $V_3(G) \neq \emptyset$. 

Let $L=(l_1,l_2,\dots,l_m) \subseteq \mathbb R$ be a positive increasing arithmetic sequence of length $m$ with common difference $d$, $d > 0$. Set $L_s=L \setminus L_b=(l_1,l_2, \dots , l_{i_b-1}),$ where $L_b=(l_{i_b},\dots,l_m)$ is the set of big labels defined in Step 3. We show that the bijection $\phi : E(G) \rightarrow L$ constructed with Algorithm 1 produces an antimagic $L$-labeling, if $L\subseteq \mathbb{R}^+$, and 
a product antimagic $L$-labeling, if $L\subseteq \mathbb{R}_\circ$.

If the condition of Line 6 in Algorithm 2 holds,
let $u$ be the leaf incident with the edge with label 1. 
Thus,  $G-u$ is a tree of size $m-1$ labeled with the elements of the positive increasing sequence $(l_2, \dots , l_{m})\subseteq
\mathbb{R}_{\circ}\setminus \{1\}$. Moreover, if $\phi$ restricted to $E(G-u)$ is a product antimagic $L\setminus\{1\}$-labeling of $G-u$, then obviously $\phi$ is a product antimagic $L$-labeling of $G$. Hence, in this case we assume $G:=G-u$ for the remaining part of the proof.

We begin by proving some properties related with the labels of $S$-edges. Recall that $S$-edges are the edges of the subgraphs in the sequence $S$ 
 
 $$S=(P_{1,1},\dots, P_{1,n_1},\dots ,P_{t_1,1}\dots, P_{t_1,n_{t_1}},C_1',\dots ,C_{t_2}')$$ defined in Step 3, where $P_{i,j}$ is a path, for every $i\in \{1,\dots ,t_1\}$ and $j\in \{1,\dots ,n_{i}\}$, and $C_i'$ is an even graph for every $i\in\{1,\dots, t_2\}$ are labeled in Step 3.

\begin{claim}\label{fact1} Let $E_H=(e_1,e_2,\dots, e_n)$ for some $H \in S.$ According to the labeling process in Algorithm 2, we have
\begin{enumerate}[(1)] 
\item For every $j\in \{1,\dots, n-2\}$
\begin{enumerate}
\item if $\phi(e_j) \in L_b$, then $\phi(e_{j+2})=\phi(e_j)+d \in L_b$;
\item if $\phi(e_j) \in L_s$, then $\phi(e_{j+2})=\phi(e_j)+kd \in L_s$ for some $k \geq 1.$
\end{enumerate}
\item For every $j\in \{ 1,\dots ,n-1\}$,
if $\phi(e_j) \in L_b$ (resp. $L_s$), then $\phi(e_{j+1}) \in L_s$ (resp. $L_b$).
\end{enumerate}
\end{claim}
\begin{proof}
    {It is a direct consequence of the way $S$-edges are labeled in Step 3 with alternate labels from $L_s$ and $L_b$, and taking into account that Condition $Q(w)$ possibly applies.}
\end{proof}

\begin{claim}\label{fact2} Every interior vertex of $G$ is an end-vertex of an $S$-edge with a label in $L_b.$
\end{claim}
\begin{proof}
    {Let $v$ be an interior vertex of $G$. Thus, $v$ belongs to the pruned graph $\mathcal{P}(G)$, since $v$ is not a leaf. Hence, $v$ belongs  to $G_1$ or $G_2$. 
    If $v\in V(G_2)$, then at least one of the $S$-edges incident with $v$ is labeled with a label from $L_b$ in Step 3.
    If $v\notin V(G_2)$ and  $v$ is an interior vertex of one of the paths in $\{P_{1,1},\dots, P_{1,n_1},\dots,P_{t_1,1},\dots ,P_{t_1,n_{t_1}}\}$, then at least one of the $S$-edges incident with $v$ is labeled with an element from $L_b$ in Step 3.

    In any other case,  $E(G_2)=\emptyset$ and $G$ is a tree, so that {$C_1=G_1=\mathcal{P}(G)$}, and $v$ is the center of the GPD of $C_1$.
    But then, the {$S$-edge} of $P_{1,1}$ incident with $v$ is labeled with an element from $L_b$ in Step 3 (see Algorithm 2). 
   }
\end{proof}

\begin{claim}\label{fact3} Every vertex of degree two {in $G$} is saturated in Step 3. 
\end{claim}
\begin{proof}
    It follows from the fact that both edges incident with a vertex $v$ of degree two in $G$ are $S$-edges and are labeled in Step 3.
\end{proof}

\begin{claim}\label{fact4} Every edge of $G_1$ incident with a vertex of $G_2$ is labeled before any edge of $G_2$.
\end{claim}
\begin{proof}
    {It is a direct consequence of  the chosen order to label $S$-edges in Step 3.}
\end{proof}

 The following claims show some properties about the vertex sum and vertex product at saturated and almost saturated vertices.

\begin{claim}\label{Almost sat. sum}
Let $v \in V_3(G)$ and let $\ell_v$ be the biggest used label in $L_b$ just when $v$ achieves the almost saturated state. Then $\phi_{\oplus}^*(v) \geq \ell_v.$ 
\end{claim}

\begin{proof}
Let $e$ be the edge such that $\phi(e)=\ell_v$. 

If $e$ is incident with $v$, then the result is clear.  
Notice that this happens in particular when $v \in V(G_2).$ 
Indeed, consider an Eulerian circuit $E_H$ {of the component} of $G_2$ containing $v$.
If $v$ is the end-vertex of  $E_H$, then,   by Claim \ref{fact4}, $v$ achieves the almost saturated state when we arrive at $v$ with the last edge of $E_H$ with a label in $L_b$.
Otherwise, also by Claim \ref{fact4}, $v$ achieves the almost saturated state when we leave this vertex with an edge of $E_H$. Thus, the label of this edge or the preceding one belongs to $L_b$.

Hence, it is enough to consider now the case $v \notin  V(G_2)$ and $e$ is not incident with $v$. Thus, $v$ is an interior vertex of $G_1$. Suppose that $v$ achieves the almost saturated state just when an edge $\tilde e$ incident with $v$ receives a label $l \in L_s,$ where $l=l_1+k'd$, for some $k'\geq 0$. 
By Claim~\ref{fact2}, there exists an edge $e'$ incident with $v$ such that $\phi(e') \in L_b.$ Notice that $\phi(e') < \phi(e).$ Besides, by construction,  the number of labels from $L_b$ already used at this moment is at most the number of used labels from $L_s$.  Then $\phi(e)-\phi(e') =kd$, where $k\le k'$. Hence, $\phi(e)-\phi(e') =kd <   l_1+k'd= l,$ since $l_1>0$. 

Therefore, 
$$\begin{array}{rlr}
\phi_{+}^*(v)\geq&\hspace{-2mm}\phi(e')+\phi (\tilde e) = \phi(e')+l >  \phi(e)=l_v, &\hbox{ if } \oplus = +;\\

\phi_{\circ}^*(v)\geq&\hspace{-2mm} \phi(e') \, \phi (\tilde e) = \phi(e')\, l =(\phi(e)-kd)(l_1+k'd) \\ \geq&\hspace{-2mm}(\phi(e)-kd)(l_1+kd)=
\phi(e)\, l_1+kd\, (\phi(e)-l_1-kd)&
\\=&\hspace{-2mm}\phi(e)\, l_1+kd\, (\phi(e')-l_1)>   \phi(e)=l_v, &\hbox{ if } \oplus=\circ.  
\end{array}$$

\end{proof}

Now we define the following subsets of $V(G)$:
\begin{align*}
W_1&= \{ v \in V(G) : \deg(v)=1 \},\\
W_2&= \{ v \in V(G) : \deg(v)= 2 \hbox{ or } ( \deg(v) > 2 \text{ and $v$ reaches the saturated state in Step 3} ) \},\\
W_3&= \{ v \in V(G) : \deg(v) > 2 \text{ and } v \text{ reaches the saturated state in Step 4} \}.
\end{align*}

Taking into account Claim \ref{fact3}, we have that $W_2$ consists of all non leaves vertices that reach the saturated state in Step 3. So, $(W_1,W_2,W_3)$ is a partition of $V(G).$

\begin{claim}
If $v \in W_2,$ then

$$\begin{array}{ll}
 \phi_{+}(v)  \leq l_m + l,&\textrm{ if } \oplus=+; \\
 \phi_{\circ}(v)  \le   l_m \, l,&\textrm{ if } \oplus=\circ.
\end{array}$$

Where $l$ is the last label in $L_s$ assigned in Step~3.
\end{claim}  
\begin{proof}
Recall that $l_m$ is the biggest label in $L$. Let $v \in W_2$ and  $l$ be the last label in $L_s$ assigned in Step 3. If $\deg (v)=2,$ 
then the two edges incident with $v$ are $S$-edges. So, by Claim~\ref{fact1}(2), 

$$\begin{array}{ll}
  \phi_{+}(v) \leq l_m + l,&\hbox{ if } \oplus=+;\\ 
  \phi_{\circ}(v) \leq  l_m \, l,&\hbox{ if } \oplus=\circ.
\end{array}$$

Otherwise, {\bf Condition $Q(v)$} holds in Step~3, so 

$$\begin{array}{ll}
  \phi_{+}(v) \leq \phi_{+}^*(v)+l\le l_m + l,&\hbox{ if } \oplus=+;\\ 
  \phi_{\circ}(v) \leq  \phi_{\circ}^*(v)\,\, l\le  l_m\, l,&\hbox{ if } \oplus=\circ.
\end{array}$$
\end{proof}

\begin{claim}\label{Sat claim}
Let $v_1,v_2 \in W_2$. If $v_1$ reaches the saturated state before $v_2$, then $\phi_{\oplus}(v_1) < \phi_{\oplus}(v_2)$.
\end{claim}

\begin{proof}
The analysis splits into four cases.
\vspace{1.5mm}

\noindent
{\em Case 1:} $\deg(v_1)=\deg(v_2)=2.$ It follows  from Claims~\ref{fact1} and \ref{fact3}, and the order $S$-edges are labeled in Step 3.
\vspace{1.5mm}

\noindent
{\em Case 2:} $\deg (v_1)=2$ and $\deg(v_2) > 2.$ Let $e$ and $e'$ be the edges incident with $v_1$.
By Claim~\ref{fact1}(2), we can assume 
$\phi(e)=l \in L_s$ and $\phi(e')=l' \in L_b.$
By construction,  we have $\phi_{\oplus}^*(v_2) > l' - d$, otherwise, {\bf Condition $Q(v_2)$} holds and $v_2$ would be saturated before $v_1$.
Therefore, if $\oplus=+,$  
then, 

$$ \phi_{+}(v_1)= l + l' < l + d + \phi_{+}^*(v_2)  \leq  \phi_{+}(v_2).$$ 
and if $\oplus=\circ,$ 
then, 

$$\phi_{\circ}(v_1)= l \, l' {\le}  (l+d)\, (l'-d) < (l + d) \, \phi_{\circ}^*(v_2) \leq \phi_{\circ}(v_2),$$
where the first inequality of the preceding expression holds because $l'-l\ge d$ and hence

$$(l+d)\, (l'-d)=l\, l'+d\,( l'-l-d)
\ge l\, l'.$$

\noindent
 {\em Case 3:} $\deg(v_1) > 2$ and $\deg(v_2) = 2.$  It means that {\bf Condition $Q(v_1)$} holds at some point of Step 3.
 
 If $\oplus=+$, assume $\phi_{+}(v_2)=l + l',$ where $l \in L_s$ and $l' \in L_b$. Then, $\phi_{+}(v_1)= l'' + \phi_{+}^*(v_1),$ where $l'' < l$, and   $\phi_{+}^*(v_1) \leq l',$
 since $v_1$ was saturated before $v_2$. Hence 
 
 $$\phi_{+}(v_1)=l'' + \phi_{+}^*(v_1) < l + l' = \phi_{+}(v_2).$$
 
 Similarly, if $\oplus=\circ$, 
 assume $\phi_{\circ}(v_2)=l \, l',$ where $l \in L_s$ and $l' \in L_b$. 
 Then, $\phi_{\circ}(v_1)= l'' \, \phi_{\circ}^*(v_1),$ where $l'' < l$,
 and  $\phi_{\circ}^*(v_1) \leq l',$ 
 since $v_1$ was saturated before $v_2$.
 Hence 
 
 $$\phi_{\circ}(v_1)=l'' \, \phi_{\circ}^*(v_1) < l \, l' = \phi_{\circ}(v_2).$$ 

\noindent
{\em Case 4:} $\deg(v_1)>2$ and $\deg(v_2)> 2.$ It means that {\bf Condition $Q(v_1)$} and {\bf  Condition $Q(v_2)$} hold at some point of Step 3. Hence, we have $\phi_{\oplus}^*(v_1) \leq \phi_{\oplus}^*(v_2).$ Besides, if $i\in \{1,2\}$, then

$$\begin{array}{ll}
\phi_{\oplus}(v_i)=\phi_{+}(v_i)=\phi_{+}^*(v_i) + l_{i}\, ,&\hbox{ if }\oplus=+,\\
\phi_{\oplus}(v_i)=\phi_{\circ }(v_i)=\phi_{\circ }^*(v_i)\,\, l_i,&\hbox{ if }\oplus=\circ,
\end{array}$$
where $l_{i} \in L_s$ and $l_1 < l_2$.
Consequently,  $\phi_{\oplus}(v_1) < \phi_{\oplus}(v_2)$.

\end{proof}
At any stage of the labeling process we use the smallest unused label in $L_s$ or $L_b$. Moreover, $\vert L_s\cup L_b\vert=\vert E(G) \vert$ and the number of labels in $L_s$ and $L_b$ is calculated so that in both sets we have enough labels to assign according at every step of the algorithm.
{Besides, in Steps 1 and 2 almost all pendant edges in $G$ {and} $G_1$ are labeled. After that, just before  beginning Step 3, an unlabeled edge is an $S$-edge or the only  pendant edge (in $G$}) incident with a vertex in $V_3$ that remains unlabeled. But all $S$-edges are labeled in Step 3, and the remaining pendant edges are labeled in Step 3 or in Step 4.
Hence, $\phi$ is a bijection.

For $V_1,V_2 \subseteq V(G)$, we say that $\phi_{\oplus}(V_1) < \phi_{\oplus}(V_2)$ if for every $v_1 \in V_1$ and $v_2 \in V_2$ we have $\phi_{\oplus}(v_1) < \phi_{\oplus}(v_2)$.
To show that  the labeling $\phi$ is $\oplus$-antimagic it is enough to prove that 
$\phi_{\oplus}(W_1) < \phi_{\oplus}(W_2) < \phi_{\oplus}(W_3) $
and $\phi_{\oplus}$ restricted to $W_i$ is injective, for every $i\in \{1,2,3\}$.

By construction, $\phi_{\oplus}(W_1) \subseteq L_s \cup \{l_{i_b}\}$.
Notice that the exceptional case $\phi_{\oplus}(v)=l_{i_b}$ for $v \in W_1$ 
occurs when $G$ is a tree and the first path of the GPD of $\mathcal{P}(G)$ is of odd size, 
and ($\oplus=+$ or $l_1 >1$).
Claim~\ref{fact2} implies 

$$\phi_{\oplus}(W_1) < \phi_{\oplus}(W_2\cup W_3).$$ 

Now, if $v \in W_3$, then {\bf Condition $Q(v)$} does not hold in Step 3. Thus, 
$\phi_{\oplus}^*(v) > l_m - d$, since otherwise $\phi_{\oplus}^*(v) \le  l_m - d$ and {\bf Condition $Q(v)$} would hold when assigning the last label of $L_b$, that is assigned in Step 3, and $v$ would be saturated in Step 3, contradicting that $v\in W_3$.
Therefore, if $l$ is the last label from $L_s$ assigned in Step 3, we derive

$$\begin{array}{ll}
\phi_{\oplus}(v) \geq \phi_{+}^*(v) + (l+d) > (l_m-d) + (l+d)\geq l_m + l, &\hbox{ if } \oplus=+,\\
\phi_{\oplus}(v) \geq \phi_{\circ}^*(v) \,\, (l+d) > (l_m-d)  (l+d)=
l_m \,\, l +d(l_m-l-d)
\geq l_m \, l, &\hbox{ if } \oplus=\circ.
\end{array}$$

Hence, by Claim~\ref{Almost sat. sum}, we have

$$\phi_{\oplus}(W_2) < \phi_{\oplus}(W_3).$$

It remains to prove that $\phi_{\oplus}$ restricted to each $W_i$ is injective.  Since the vertices of $W_1$ are leaves and $\phi$ is a bijection, we derive that $\phi_{\oplus}$ restricted to $W_1$ is injective.
Claim \ref{Sat claim} implies that  $\phi_{\oplus}$ restricted to $W_2$ is injective.
Finally, by construction (see Step 4 in  Algorithm 1), it is clear that the vertices in $W_3$ have distinct images under $\phi_{\oplus}$. 
Therefore, we conclude that Theorem~\ref{Antimagic no bald} holds.

\subsection{Concluding remarks}\label{sec:concluding}

By a {\em subdivision} of a graph $G$ we mean any graph 
obtained from $G$ by replacing every edge $e=uv$ 
with a path of length at least one with end-vertices $u$ and $v$.

It is worth emphasizing that if a graph satisfies $V_3\subseteq V_s$, then we can add pendant edges to it and subdivide its edges in a suitable way holding the antimagicness property at the modified graph. {More precisely}, next result follows immediately from Theorem~\ref{Antimagic no bald}.

\begin{corollary}\label{Antimagic subd no bald leafy}
If $G$ is a graph 
 such that $V_3\subseteq V_s$ and $\widetilde{G}$ is a  leafy graph of $G$, then any subdivision of $\widetilde{G}$ is 
arithmetic antimagic and arithmetic product antimagic, provided that no pendant edge of $\widetilde{G}$ is subdivided. 
\end{corollary}

Since  leafy cycles and caterpillars (leafy paths, indeed) satisfy $V_3\subseteq V_s$, we have the following  results as a consequence of Theorem~\ref{Antimagic no bald}.

\begin{corollary}
Leafy cycles are arithmetic antimagic and arithmetic product antimagic.
\end{corollary}

\begin{corollary}\label{caterpillars}
Caterpillars are arithmetic antimagic and arithmetic product antimagic.
\end{corollary}

It should be noted that, in particular, Corollary~\ref{caterpillars} states that caterpillars are antimagic and product antimagic, which are the main results obtained in  \cite{LMST} and  \cite{WG}, respectively.
\vspace{1mm}

Considering the techniques used to prove that a graph is antimagic, our feeling is that  antimagicness implies arithmetic antimagicness.

\begin{open} Prove or disprove that every  antimagic graph is arithmetic antimagic.
\end{open}


\bibliographystyle{abbrv} 
\bibliography{bibfile}

\end{document}